\documentclass[12pt,a4paper]{amsart}

\usepackage{amscd,amssymb,amsopn,amsmath,amsthm,mathrsfs,graphics,amsfonts,enumerate,verbatim,calc
}
\usepackage[all]{xypic}
\usepackage{url}
\usepackage{color}
\usepackage{hyperref}
\usepackage{cleveref}
\usepackage{stmaryrd}

\usepackage[OT2,OT1]{fontenc}
\newcommand\cyr{%
\renewcommand\rmdefault{wncyr}%
\renewcommand\sfdefault{wncyss}%
\renewcommand\encodingdefault{OT2}%
\normalfont
\selectfont}
\DeclareTextFontCommand{\textcyr}{\cyr}

\usepackage{amssymb,amsmath}

\DeclareFontFamily{OT1}{rsfs}{}
\DeclareFontShape{OT1}{rsfs}{n}{it}{<-> rsfs10}{}
\DeclareMathAlphabet{\mathscr}{OT1}{rsfs}{n}{it}

\numberwithin{equation}{section}
\newtheorem{theorem}{Theorem}[section]
\newtheorem{lemma}[theorem]{Lemma}
\newtheorem{proposition}[theorem]{Proposition}
\newtheorem*{proposition*}{Proposition}
\newtheorem{corollary}[theorem]{Corollary}
\newtheorem*{maintheorem}{Main Theorem}

\theoremstyle{definition}
\newtheorem{definition}[theorem]{Definition}
\newtheorem{remark}[theorem]{Remark}

\newtheorem{Notation}[theorem]{Notation}
\theoremstyle{remark}

\newtheorem{example}[theorem]{Example}
\newtheorem*{acknowledgements}{Acknowledgements}

\newcommand{\Spec}{\operatorname{Spec}}

\newcommand{\fm}{\mathfrak{m}}
\newcommand{\fp}{\mathfrak{p}}
\newcommand{\fq}{\mathfrak{q}}

\newcommand{\mcQ}{\mathcal{Q}}
\newcommand{\mcP}{\mathcal{P}}

\newcommand{\bprop}{\begin{prop}}
\newcommand{\eprop}{\end{prop}}
\newcommand{\bthem}{\begin{them}}
\newcommand{\ethem}{\end{them}}
\newcommand{\blem}{\begin{lem}}
\newcommand{\elem}{\end{lem}}
\newcommand{\bdefn}{\begin{definition}}
\newcommand{\edefn}{\end{definition}}
\newcommand{\bpr}{\begin{proof}}
\newcommand{\epr}{\end{proof}}
\newcommand{\bex}{\begin{example}}
\newcommand{\eex}{\begin{example}}
\newcommand{\bit}{\begin{itemize}}
\newcommand{\eit}{\end{itemize}}
\newcommand{\ben}{\begin{enumerate}}
\newcommand{\een}{\end{enumerate}}
\newcommand{\bprob}{\begin{prob}}
\newcommand{\eprob}{\end{prob}}
\newcommand{\bco}{\begin{cor}}
\newcommand{\eco}{\end{cor}}

\newcommand{\gp}{\operatorname{gp}}

\begin{document}

\title{Log-regularity of monoid algebras}

\author{Shinnosuke Ishiro}
\address{National Institute of Technology, Gunma College, 580 Toriba-machi, Maebashi-shi, Gunma 371-8530, Japan}
\email{shinnosukeishiro@gmail.com}

\begin{abstract}
    Local log-regular rings are Cohen--Macaulay local domains introduced by Kazuya Kato to expand the theory of toric varieties without a base.
    In this note, we show that local rings of monoid algebras over regular rings are log-regular.
    As a corollary, we show that monoid algebras are log-regular rings.
\end{abstract}

\maketitle

\section{Introduction}

In order to explore arithmetic geometry and commutative algebra in mixed characteristic, one needs a generalization of classes of commutative rings over fields in the traditional theory.
K. Kato introduced local log-regular rings and log-regular schemes to apply the theory of toric geometry to arithmetic geometry (\cite{Kat94}).
Local log-regular rings are a certain class of Noetherian local rings equipped with a structure given by a monoid homomorphism (see \Cref{LogRegDef}).
This class of rings has similar properties to normal affine semigroup rings, for instance, they are Cohen--Macaulay normal domains.
Several results of local log-regular rings, including the structure of canonical modules and of the divisor class groups, can be found in the author's paper \cite{Ish25}.
The relation with perfectoid theory is contained in \cite{GRBook24} or \cite{INS25}.
For a treatment in algebraic and arithmetic geometry, we refer the reader to \cite{OgusBook}.
Many properties of local log-regular rings are induced by an analogue of Cohen's structure theorem.
Namely, if a local log-regular ring is complete, then the ring is isomorphic to a complete monoid algebra (if the ring is of equal characteristic) or a homomorphic image of a complete monoid algebra (if the ring is of mixed characteristic).
Local log-regular rings naturally appear in the study of the resolution of singularities.
An example of local log-regular rings appeared in Abhyankar's study of resolution of arithmetic surfaces (\cite{Abh65}), which is called \textit{Jungian domains}.
In recent work, log-regularity plays a crucial role in Gabber's weak local uniformization (\cite{ILO14}).

As mentioned above, local log-regular rings frequently appear in the study of commutative algebra and arithmetic geometry, but examples of non-complete local log-regular rings are rather scarce in the literature, except for Jungian domains.
In view of Kato's introduction of log-regularity, it is natural to ask whether localizations of monoid algebras are log-regular.
However, it is not immediate whether localizations of monoid algebras are log-regular, and to the best of our knowledge, this question has not been explicitly addressed in the literature.
In this note, we prove the following theorem, which clarifies this behavior and provides new examples of non-complete local log-regular rings:

\begin{maintheorem}\label{MainTheorem}
    Let $\mathcal{Q}$ be a cancellative, finitely generated, and root closed monoid\footnote{A monoid which satisfies the above properties is also called a positive normal affine semigroup ring.} such that the quotient group $\mathcal{Q}^{gp}$ is torsionfree.
    Let $A$ be a regular ring (that is not necessarily local).
    Let $R := A[\mathcal{Q}]$ be the monoid algebra.
    Then, for any prime ideal $\fp \subset R$ with its contraction $P := \fp \cap \mathcal{Q}$, the triple $(R_\fp, \mathcal{Q}_P, \iota_\fp : \mathcal{Q}_P \hookrightarrow R_\fp)$ is a local log-regular ring.
\end{maintheorem}

Main Theorem provides concrete examples of non-complete local log-regular rings.
Let $\mathbb{Z}$ be the ring of integers and let $\mcQ \subset \mathbb{N}^4$ be a monoid generated by four elements 
\[
(1,1,0,0), (0,0,1,1), (1,0,1,0), (0,1,0,1).
\]
Then $\mathbb{Z}[\mcQ]$ is isomorphic to $\mathbb{Z}[x,y,z,w]/(xy-zw)$.
For a prime ideal $\fp \in \Spec(\mathbb{Z}[\mcQ])$, we obtain that $\mathbb{Z}[\mcQ]_\fp$ is a local log-regular ring.

As a corollary, a monoid algebra $R[\mcQ]$ appearing in Main Theorem is a non-local log-regular ring (\Cref{LogRegMonAlg}).
Log-regular rings realize log-regular schemes and log-regular schemes play central roles in arithmetic geometry.
For instance, W. Nizio\l ~\cite{Niz06} established a desingularization of log-regular schemes.
We expect that our result promotes understanding of log-regularity.

\section{Proof of Main Theorem}
In this note, a \textit{monoid} means a commutative semigroup with a unity $0$. Unless otherwise noted, we denote the binary operation of monoids by $+$.
Rings are also assumed to be commutative.
When considering a ring as a monoid, its monoid structure is always multiplicative.

\begin{Notation}
    Let $\mathcal{Q}$ be a monoid.
    We denote the group of units of a ring $R$ (resp. a monoid $\mcQ$) by $R^\times$ (resp. $\mcQ^*$).
    We denote $\mathcal{Q}\setminus\mcQ^*$ by $\mcQ^+$.
\end{Notation}

\begin{definition}
    Let $\mathcal{Q}$ be a monoid.
    \begin{enumerate}
        \item A subset $I$ of $\mcQ$ is an \textit{ideal} if $a+x \in I$ for any $a \in \mcQ$ and $x \in I$.
        \item An ideal $\fp$ of $\mcQ$ is called \textit{prime} if $\fp \neq \mcQ$ and $a+b \in \fp$ implies $a \in \fp$ or $b \in \fp$ for any $a,b \in \fp$.
        \item A subset $F$ of $\mcQ$ is called a \textit{face} if there exists a prime ideal $\fp$ such that $F = \mcQ \setminus \fp$.
    \end{enumerate}
\end{definition}

It is easy to show that a subset $F$ of a monoid $\mcQ$ is a face if and only if it is a submonoid and $p+q \in F$ implies $p, q \in F$.
By using this, one can prove that a face of a finitely generated monoid is also finitely generated.
We also have a one-to-one correspondence between the set of prime ideals and the set of faces.
This is formulated the map that sends $\fp$ to $F_{\fp} :=\mathcal{Q}\setminus\fp$.


We recall some basic properties of monoids.
For details, we refer to \cite{GFBook06}, \cite{GilBook84}, and \cite{OgusBook}.

\begin{definition}\label{MonPropDef}
Let $\mathcal{Q}$ be a monoid.
\begin{enumerate}
    \item $\mathcal{Q}$ is called \textit{cancellative} if for $x, x'$ and $y \in \mcQ$, $x+y=x' +y$ implies $x=x'$.
    \item $\mathcal{Q}$ is called \textit{reduced} if $\mcQ^{*}=0$.
    \item $\mcQ$ is called \textit{root closed} if it satisfies the following conditions.
        \begin{itemize}
        \item $\mcQ$ is cancellative.
        \item If $x \in \mcQ^{\gp}$ such that $nx \in \mcQ$ for some $n >0$, then $x \in \mcQ$.
        \end{itemize}
\end{enumerate}
\end{definition}

\begin{definition}
Let $\mcQ$ be a cancellative monoid.
Then we define the \textit{quotient group of $\mathcal{Q}$} as the localization at $(0)$.
In particular, elements of the group are of the form $a-b$, where $a,b \in \mcQ$.
We denote it by $\mathcal{Q}^{gp}$.
\end{definition}

\begin{remark}
Terminology for monoids differs slightly between logarithmic geometry and classical monoid theory.
Throughout this paper, we follow the terminology standard in classical monoid theory; in particular, a cancellative, reduced, and root closed monoid corresponds to what is often called integral, sharp, and saturated, respectively, in the literature on logarithmic geometry.

In addition, a face is also called \textit{divisor-closed}, and $\mathcal{Q}^{gp}$ is also called the \textit{Grothendieck group of $\mcQ$} and denoted by $\mathbf{q}(\mcQ)$. These are used in \cite{GFBook06}. 
\end{remark}

The localization of a monoid at a prime ideal is defined in the same way as that of a ring.
The following lemma explains some of the properties of a monoid that are inherited by each of its localizations.

\begin{lemma}\label{LocalizationProp}
    Let $\mathcal{Q}$ be a monoid and let $\fp \subseteq \mathcal{Q}$ be a prime ideal.
    Denote by \textbf{P} any of the properties cancellative, finitely generated, or root closed.
    If $\mathcal{Q}$ satisfies $\textbf{P}$, then so does $\mathcal{Q}_{\fp}$.
\end{lemma}

\begin{proof}
    Note that $F_\fp := \mcQ\setminus\fp \subseteq \mcQ$ is finitely generated.
    Also, the localization by a finitely generated submonoid is also finitely generated.
    Hence $\mathcal{Q}_{\fp}$ is finitely generated.

    Suppose that $\mathcal{Q}$ is cancellative.
    Let us consider the equality in $\mathcal{Q}_{\fp}$
    \[
    (a-b)+(a''-b'') = (a'-b') + (a''-b'').
    \]
    Then, there exists an element $c \in \mathcal{Q}$ such that the equality in $\mathcal{Q}$
    \[
    c+ (a+a'')+(b'+b'') = c+(a'+a'') + (b+b'').
    \]
    Since $\mathcal{Q}$ is cancellative, we obtain 
    \[
    a+b' = a'+b
    \]
    in $\mathcal{Q}$.
    This implies that the equality $a-b = a' - b'$ in $\mathcal{Q}_{\fp}$.
    Hence $\mathcal{Q}_{\fp}$ is also cancellative.

    Finally, suppose that $\mathcal{Q}$ is root closed.
    We already proved that $\mcQ_{\fp}$ is cancellative.
    Pick an element $a \in (\mcQ_{\fp})^{gp}$ such that $na = x-y \in \mathcal{Q}_{\fp}$ for some $n \in \mathbb{Z}_{>0}$, $x \in \mcQ$, and $y \notin \fp$.
    This implies $na +y \in \mathcal{Q}$, hence $n(a+y) = na+ny \in \mcQ$.
    Since $\mcQ$ is root closed, we obtain $a+y \in \mcQ$.
    Therefore $a \in \mcQ_\fp$, as desired.
\end{proof}

An analogue of \Cref{LocalizationProp}

\begin{definition}
    Let $R$ be a ring, let $\mcQ$ be a monoid, and let $\alpha : \mcQ \to R$ be a homomorphism of monoids, where the monoid structure of $R$ is multiplicative.
    Then the triple $(R,\mcQ, \alpha)$ is called a \textit{log ring}.
    Also, if $R$ is a local ring and $\alpha$ is a local homomorphism (i.e. $\alpha^{-1}(R^\times) = \mcQ^*$), then the log ring $(R,\mcQ, \alpha)$ is called \textit{local}. 
\end{definition}

\begin{lemma}\label{localizationlocallog}
Let $\alpha : \mcQ \to \mathcal{P}$ be a homomorphism of monoids and let $G \subseteq \mathcal{P}$ be a face.
Set $F := \alpha^{-1}(G)$.
Then the induced homomorphism $\alpha_F : \mcQ_F \to \mathcal{P}_G$ is local.
In particular, if $(R,\mcQ,\alpha)$ is a log ring and $\fp$ is a prime ideal of $R$, then $(R_\fp, \mcQ_P, \alpha_{\fp})$ is a local log ring where $P := \alpha^{-1}(\fp)$ and $\alpha_\fp(x-y) = \frac{\alpha(x)}{\alpha(y)}$.
\end{lemma}

\begin{proof}
    The second assetion follows from the first assertion immediately.
    For the first assertion, it suffices to show that $\alpha_F^{-1}(\mcP_G^*)$ is contained in $\mcQ_F^*$.
    Pick an element $x = q-f \in \alpha^{-1}_F(\mcP^*_G)$.
    Then we have an element $x' \in \mcP_G$ such that $\alpha_F(x)+x'=0$.
    Also, $\alpha_F(x) + \alpha(f)$ is contained in $\mcP$ and there exists an element $g \in G$ such that $\bigl(\alpha_F(x) +\alpha(f)\bigr) +g = \alpha(q) + g$ in $\mcP$.
    Since $\iota_G\bigl(\alpha_F(x) +\alpha(f)\bigr) + \bigl( x' -\alpha(f) \bigr) = 0$ in $\mcP_G$, we obtain $\alpha_F(x) + \alpha(f) \in \iota_G^{-1}(\mcP_G^*) = G$ where $\iota_G$ is the canonical homomorphism $\mcP \to \mcP_G$.
    Hence $q$ lies in $F$ and this implies that $x$ lies in $\mcQ_F^*$, as desired.
\end{proof}

Let $A$ be a ring and let $\mathcal{Q}$ be a monoid.
Then a \textit{monoid algebra $A[\mcQ]$ over $A$} consists of all formal finite sums $\{ \sum_{q \in \mcQ} a_q e^q |~a_q \in A \}$, where $e^q$ is the image of an element $q \in \mcQ$ in $A[\mcQ]$.

Let $\fp$ be a prime ideal of $\mcQ$ and let $F_\fp:= \mcQ \setminus \fp$ be the corresponding face at $\fp$.
Then the composition 
\begin{equation}\label{MonoidIsomFace}
A[F_\fp]\hookrightarrow A[\mcQ] \twoheadrightarrow A[\mcQ]/\fp A[\mcQ]
\end{equation}
is an isomorphism.
Indeed, the surjectivity follows from the representation $f = \sum_{q \notin \fp} a_qe^q + \sum_{q' \in \fp} a_{q'}e^{q'}$ for any $f \in A[\mcQ]$.
Fix an element $f$ of the kernel of the composition (\ref{MonoidIsomFace}).
Then we have $f = \sum_{p\in \fp} e^p(\sum_{q \in \mcQ} a_q e^q) = \sum a_q e^{q+p} \in A[F_\fp]$.
Suppose that $f$ is not zero.
Then there exist elements $p \in \fp$ and $q \in \mcQ$ such that $p+q \in F_\fp$.
This implies $p \in F_\fp$, but this is a contradiction.
Hence (\ref{MonoidIsomFace}) is injective.

In particular, we have an isomorphism
\begin{equation}\label{MonoidIsomPlus}
A[\mcQ]/\mcQ^+A[\mcQ] \xrightarrow{\cong} A[\mcQ^{*}].
\end{equation}

Next, we introduce the notion of very solidness of a log ring.
This property describes the correspondence between prime ideals of the monoid in a log ring and those of the underlying ring. 

\begin{definition}\label{DefVS}
A log ring $(R,\mcQ,\alpha)$ is \textit{very solid} if for every prime ideal $\fq$ of $\mcQ$, the ideal $\fq R$ of $R$ generated by the image of $\fq$ is also a prime ideal of $R$ such that $\alpha^{-1}(\fq R) = \fq$.
\end{definition}

\begin{example}\label{TorsionfreeMonAlgVS}
    Let $\mcQ$ be a cancellative monoid such that $\mcQ^{gp}$ is torsionfree, let $A$ be an integral domain, and let $\iota : \mcQ \hookrightarrow A[\mcQ]$ be the inclusion.
    Then a log ring $(A[\mcQ], \mcQ, \iota)$ is very solid.
    This is already proved in \cite[Corollary 8.2]{GilBook84} and also in \cite[Chapter \textbf{III} Proposition 1.10.12 (1)]{OgusBook}, but we give a proof here again.
    Let $\fp$ be a prime ideal of $\mcQ$.
    By (\ref{MonoidIsomFace}), we know that $A[\mcQ]/\fp A[\mcQ] \cong A[F_\fp]$ where $F_\fp$ is the corresponding face at $\fp$.
    Since $A$ is an integral domain and $F_\fp$ is cancellative, the monoid algebra $A[F_\fp]$ is also an integral domain.
    Hence $\fp A[\mcQ]$ is a prime ideal.
    Also, since the equalities $\iota^{-1}(\fp A[\mcQ]) = \fp \cap \mcQ = \fp$ holds, the latter condition is satisfied.
\end{example}

\begin{remark}\label{NonVSRemark}
    In the setting of \Cref{TorsionfreeMonAlgVS}, if $A$ is not an integral domain, then $(A[\mcQ],\mcQ, \iota)$ is not very solid.
    Indeed, we also know that $A[\mcQ]/\fp A[\mcQ]$ is isomorphic to $A[F_\fp]$ where $F_\fp$ is a corresponding face at $\fp$.
    However, since $A[F_\fp]$ is not an integral domain, $\fp A[\mcQ]$ is not a prime ideal.
    This implies that $(A[\mcQ], \mcQ, \iota)$ is not very solid.
\end{remark}

\begin{lemma}\label{LocaliryVS}
    Let $(R,\mathcal{Q},\alpha)$ be a log ring.
    If $(R,\mathcal{Q},\alpha)$ is very solid, then the induced log ring $(R_\fp, \mathcal{Q}_P, \alpha_\fp)$ is very solid for any prime ideal $\fp$ of $R$, where $P = \fp \cap \mathcal{Q}$.
\end{lemma}

\begin{proof}
    Let $\fq$ be a prime ideal of $\mcQ$.
    Consider the short exact sequence
    \[
    \xymatrix{
    0 \ar[r] & (\fq R)_{\fp} \ar[r] & R_\fp \ar[r] & (R/\fq R)_{\fp} \ar[r] & 0.
    }
    \]
    Note that we have the equality $(\fq R)_\fp = \fq_P R_\fp$ where $\fq_P$ is a prime ideal of $\mcQ_P$ corresponding to $\fq$.
    Moreover, $(R/\fq R)_\fp$ is a domain, which implies that $\fq_P R_\fp$ is a prime ideal.
    Also, we have
    \[
    \alpha^{-1}_{\fp}(\fq_PR_\fp) = \alpha^{-1}_{\fp}((\fq R)_\fp) = \alpha^{-1}(\fq R)_P = \fq_P.
    \]
    Therefore, the assertion holds.
\end{proof}

Let $\mcQ$ be a monoid.
Two elements $a, b \in \mcQ$ are called \textit{associates} if there exists a unit $u \in \mcQ^*$ such that $a=u+b$.
If $a, b \in \mcQ$ are associates, we write $a \simeq b$.
The relation $\simeq$ is a congruence relation and the monoid $\mcQ_{\textnormal{red}} := \mcQ/\simeq$ is called the \textit{associated reduced monoid of $\mcQ$}.
One can easily prove that if $\mcQ$ is cancellative (resp. finitely generated, root closed) then so is $\mcQ_{\textnormal{red}}$.
See \cite[Prerequisites]{GFBook06} and \cite[Chapter \textbf{I}, Proposition 1.3.3 and Proposition 1.3.5]{OgusBook}.
We now recall the notion of log-regularity for commutative rings.
\begin{definition}\label{LogRegDef}
    Let $(R,\mcQ,\alpha)$ be a local log ring, where $R$ is Noetherian and $\mcQ$ is cancellative such that $\mcQ_{\textnormal{red}}$ is finitely generated and root closed.
    We denote by $I_\alpha$ the ideal of $R$ generated by $\alpha(\mathcal{Q}^+)$.
    Then $(R,\mcQ,\alpha)$ is called \textit{log-regular} if it satisfies the following two conditions.
    \begin{enumerate}
        \item $R/I_\alpha$ is a regular local ring.
        \item The equality $\dim R = \dim R/I_\alpha + \dim \mcQ$ holds.
    \end{enumerate}
\end{definition}

\begin{remark}\label{LogRegEquiv}
    In the setting of \Cref{LogRegDef}, if the local log ring $(R, \mcQ, \alpha)$ satisfies the condition (1) (i.e. $R/I_\alpha$ is regular), then the condition (2) is equivalent to the condition that $(R, \mcQ, \alpha)$ is very solid.
    This assertion is stated in \cite[Chapter \textbf{III} Theorem 1.11.1]{OgusBook}.
\end{remark}

Here we restate and prove the main theorem.

\begin{theorem}\label{maintheorem1}
    Let $\mathcal{Q}$ be a cancellative, finitely generated, and root closed monoid such that the quotient group $\mathcal{Q}^{gp}$ is torsionfree.
    Let $A$ be a regular ring that is not necessarily local.
    Let $R := A[\mathcal{Q}]$ be the monoid algebra.
    Then, for any prime ideal $\fp \subset R$ with its contraction $P := \fp \cap \mathcal{Q}$, the triple $(R_\fp, \mathcal{Q}_P, \iota_\fp : \mathcal{Q}_P \hookrightarrow R_\fp)$ is a local log-regular ring.
\end{theorem}

\begin{proof}
    Since $R_\fp$ is a localization of $A_{\fp\cap A}[\mcQ]$ that is the monoid algebra over the regular local ring $A_{\fp \cap A}$, we may assume that $A$ is a regular local ring.
    In particular, $A$ is a domain.
    We note that $\mcQ_P$ is cancellative, finitely generated, and root closed by \Cref{LocalizationProp}.
    By Lemma \ref{localizationlocallog}, we know that $(R_\fp, \mcQ_P, \iota_\fp)$ is a local log ring.
    Moreover, since $\mcQ_P$ is cancellative, finitely generated, and root closed, so is $(\mathcal{Q}_P)_{\textnormal{red}}$.
    Therefore, it suffices to show that the local log ring $(R_\fp, \mathcal{Q}_P, \iota_\fp)$ satisfies the conditions (1) and (2) in \Cref{LogRegDef}.
    
    Let $I_\fp$ be the ideal generated by the image of $P$ under $\iota_\fp$.
    We have the inclusion maps
    \begin{equation}\label{compositelocalization1}
    R= A[\mcQ] \hookrightarrow A[\mcQ_P] \hookrightarrow R_{\fp}.
    \end{equation}
    This yields the injective ring homomorphisms
    \begin{equation}\label{compositelocalization2}
    R/R\cap I_{\iota_\fp} \hookrightarrow A[\mcQ_P]/A[\mcQ_P]\cap I_{\iota_\fp} \hookrightarrow R_\fp/I_{\iota_\fp}.
    \end{equation}
    Moreover, by taking the localizations, we obtain
    \begin{equation}\label{LocalizationHookArrows}
    (R/R\cap I_{\iota_\fp})_\fp \hookrightarrow (A[Q_P]/A[Q_P] \cap I_{\iota_\fp})_{PA[\mcQ_P]} \hookrightarrow (R_\fp/ I_{\iota_\fp})_\fp.
    \end{equation}
    Since the far left of (\ref{LocalizationHookArrows}) is isomorphic to the far right, they are isomorphic to each other.
    Also, the equality $R\cap I_{\iota_\fp} = PA[\mcQ_P]$ holds.
    Hence we have 
    \[
    A[\mathcal{Q}_P]/A[\mathcal{Q}_P]\cap I_{\iota_\fp} = A[\mathcal{Q}_P]/PA[\mathcal{Q}_P] \cong A[(\mathcal{Q}_P)^*],
    \]
    where the isomorphism is induced by (\ref{MonoidIsomPlus}).
    Since $A[(\mathcal{Q}_P)^*]$ is isomorphic to a Laurent polynomial ring, its localizations are regular local rings.
    Therefore $(A[\mcQ]_{\fp}/I_{\iota_\fp})_{\fp} \cong A[\mcQ]_\fp/I_{\iota_\fp}$ is also regular.
    Therefore, it satisfies the condition (1).

    It remains to show that $(R_{\fp}, \mcQ_P, \iota_{\fp})$ satisfies the condition (2) in Definition \ref{LogRegDef}.
    However, $R_\fp$ is a localization of a monoid algebra over a regular local ring.
    Hence, by combining \Cref{TorsionfreeMonAlgVS} with \Cref{LocaliryVS}, we know that $(R_\fp, \mathcal{Q}_P, \alpha_\fp)$ is very solid.
    This implies that it satisfies the condition (2) by \cite[Chapter \textbf{III} Theorem 1.11.1]{OgusBook} (see also \Cref{LogRegEquiv}), as desired.
\end{proof}

By combining our main theorem with the structure theorem of complete local log-regular rings, we obtain an explicit description of the completion of a localization of a monoid algebra.
For details on the structure theorem of complete local log-regular rings, we refer the reader to \cite[Chapter \textbf{III}, Theorem 1.11.2]{OgusBook} or \cite[Theorem 2.22]{INS25}.

\begin{corollary}
    Keep the notation as in \Cref{maintheorem1}.
    Let $k_\fp$ be the residue field of $R_\fp$ and let $\widehat{R_\fp}$ be the completion of $R_\fp$ with respect the maximal ideal $\fp R_\fp$.
    Set $r = \dim R/I_{\iota_\fp}$.
    Then the following assertions hold.
    \begin{enumerate}
        \item 
        If $R_\fp$ is of equal characteristic, then we obtain the following commutative diagram
        \[
        \xymatrix{
        (\mcQ_P)_{\textnormal{red}} \ar[r] \ar[d] & k_\fp \llbracket (\mcQ_P)_{\textnormal{red}} \oplus \mathbb{N}^r \rrbracket \ar[d] \\
        R_\fp \ar[r] & \widehat{R_\fp},
        }
        \]
        where the right vertical arrow is an isomorphism.
        \item
        If $R_\fp$ is of mixed characteristic $p>0$, then we obtain the following commutative diagram
        \[
        \xymatrix{
        (\mcQ_P)_{\textnormal{red}} \ar[r] \ar[d] & C(k_\fp)\llbracket (\mcQ_P)_{\textnormal{red}} \oplus \mathbb{N}^r \rrbracket \ar[d] \\
        R \ar[r] & \widehat{R},
        }
        \]
        where $C(k_\fp)$ is the Cohen ring of $k_\fp$ and the right vertical arrow is a surjective ring map whose kernel is a principal ideal generated by an element whose constant term is $p$.
        In particular, $\widehat{R_\fp}$ is isomorphic to $C(k_\fp)\llbracket (\mcQ_P)_{\textnormal{red}} \oplus \mathbb{N}^r \rrbracket /(\theta)$ for some $\theta \in \fm_{C(k_\fp)\llbracket (\mcQ_P)_{\textnormal{red}} \oplus \mathbb{N}^r \rrbracket}$ where the constant term of $\theta$ is $p$.
    \end{enumerate}
\end{corollary}

\begin{proof}
    Replace the log structure $(R_\fp,\mcQ_P,\iota_\fp )$ with $(R_\fp, (\mcQ_P)_{\textnormal{red}}, (\iota_\fp)_{\textnormal{red}}:(\mathcal{Q}_P)_{\textnormal{red}} \to \mcQ_P \xrightarrow{\iota_\fp} R_\fp )$.
    The fact that $(R_\fp, (\mcQ_P)_{\textnormal{red}}, (\iota_\fp)_{\textnormal{red}})$ is a local log-regular ring follows from \cite[Remark 2.20]{INS25}.
    By applying the structure theorem of complete local log-regular rings to the local log-regular ring $(R_\fp, (\mcQ_P)_{\textnormal{red}}, (\iota_\fp)_{\textnormal{red}})$, we obtain the desired claims.
\end{proof}

\begin{remark}
    The main theorem provides a counterexample of the converse in \Cref{LocaliryVS}.
    Indeed, in the setting of the main theorem, suppose that $A$ is not an integral domain.
    Then, for any prime ideal $\fp$ of $\mcQ$, $(A[\mcQ]_\fp, \mcQ_P, \iota_\fp)$ is a local log-regular ring, in particular, it is very solid.
    However $(A[\mcQ], \mcQ, \iota)$ is not very solid by \Cref{NonVSRemark}.
    This implies that the converse in \Cref{LocaliryVS} does not hold.
\end{remark}


In the rest of this note, we relate our local results to the log-regularity of log schemes.
The generalities of sheaves of monoids and log schemes can be found in \cite[Chapter \textbf{II} and Chapter \textbf{III}]{OgusBook}.

Let $(R,\mcQ,\alpha)$ be a log ring.
Set $X := \Spec(R)$.
The homomorphism $\alpha$ induces a homomorphism of sheaves of monoids $\alpha_X : \mcQ_X \to \mathcal{O}_X$\footnote{In general, for a scheme $X$, a homomorphism of sheaves of monoids $\alpha_X : \mathcal{M}_X \to \mathcal{O}_X$ called a \textit{prelog structure} on $X$.}, where $\mcQ_X$ is the constant sheaf on $X$ associated to $\mcQ$ and $\mathcal{O}_X$ is the structure sheaf of $X$.
We define $\mcQ_X^{log}$ as the pushout $\mcQ_X \oplus_{\alpha_X^{-1}(\mathcal{O}_X^\times)} \mathcal{O}_X^\times$ of the diagram
\[
\xymatrix{
\alpha_X^{-1}(\mathcal{O}_X^\times) \ar[r] \ar[d] & \mcQ_X \\
\mathcal{O}_X^\times
}
\]
where $\mathcal{O}_X^\times$ is the sheaf of units of $\mathcal{O}_X$.
By construction, we obtain $\alpha_X^{log} : \mcQ_X^{log} \to \mathcal{O}_X$ that induces the isomorphism $(\alpha^{log}_X)^{-1}(\mathcal{O}_X^\times) \cong \mathcal{O}_X^\times$.
Therefre $(X,\mathcal{Q}^{log}_X,\alpha^{log}_X)$ is a log scheme\footnote{A \textit{logarithmic structure} (or \textit{log structure}, in short) $\alpha_X : \mathcal{M}_X \to \mathcal{O}_X$ on $X$ is a prelog structure such that $\alpha_X^{-1}(\mathcal{O}_X^\times) \cong \mathcal{Q}_X^\times$. A pair $(X,\mathcal{M}_X)$ (or a triple $(X,\mathcal{M}_X,\alpha_X)$) is called a \textit{log scheme}.}.

In the following, we set $\mcQ^{log}_{X,x} := (\mcQ_X^{log})_x$ and $\alpha^{log}_{X,x} := (\alpha_X^{log})_x$.
By definition we have equalities
\begin{equation}\label{StalkIsom}
\mcQ^{log}_{X,x} = \mcQ_{X,x} \oplus_{\alpha_{X,x}^{-1}(\mathcal{O}_{X,x}^\times)} \mathcal{O}_{X,x}^\times = \mcQ \oplus_{\alpha^{-1}(R_\fp^\times\cap R)} R_\fp^\times
\end{equation}
where $x:=\fp \in \Spec(R)$ and $\alpha_{X,x} : \mathcal{Q}_{X,x} \to \mathcal{O}_{X,x}$ is the homomorphism of the stalks induced by $\alpha_X$.

\begin{lemma}\label{LogarithmicUnits}
    Let $(R,\mcQ,\alpha)$ be a log ring and let $\alpha_X^{log}: \mcQ^{log}_X \to \mathcal{O}_X$ be the induced log structure on $X := \Spec(R)$.
    Fix a point $x:=\fp \in \Spec(R)$ and an element $(a,b) \in \mcQ_{X,x}^{log}$.
    Then $(a,b) \in (\mcQ_{X,x}^{log})^*$ if and only if $a \in \alpha^{-1}(R_\fp^\times \cap R)$ and $b \in R^\times_\fp$.
\end{lemma}
\begin{proof}
    Pick an element $(a,b) \in (\mcQ_{X,x}^{log})^*$.
    Then there exists an element $(u,v) \in \mcQ_{X,x}^{log}$ such that $(a,b)+(u,v) = (a+u,bv) = (0,1)$ in $\mcQ_{X,x}^{log}$.
    By the definition of the pushout, there exist elements $s, s' \in \alpha^{-1}(R_\fp^\times \cap R)$ such that $a+u+s = s'$ and $(bv)\cdot\frac{\alpha(s)}{1} = \frac{\alpha(s')}{1}$.
    Since $\alpha^{-1}(R_\fp^\times\cap R) = \mathcal{Q}\setminus\alpha^{-1}(\fp)$ is a face of $\mcQ$, the equality $a+u+s = s'$ (resp. $(bv)\cdot\frac{\alpha(s)}{1} = \frac{\alpha(s')}{1}$) implies $a \in \alpha^{-1}(R_\fp^\times \cap R)$ (resp. $b \in R_\fp^\times$), as desired.
    
    Conversely, suppose that $a \in \alpha^{-1}(R_\fp^\times \cap R)$ and $b \in R_\fp^\times$.
    Then there exists an element $a' \in R_\fp^{\times}$ and $b' \in R_\fp$ such that $\frac{\alpha(a)}{1} a' = 1$ and $bb' =1$.
    Hence we obtain equalities $(a,b) + (0,a'b') = (a,a'bb') = (0, \frac{\alpha(a)}{1}a' ) = (0,1)$, which implises that $(a,b) \in (\mcQ_{X,x}^{log})^*$.
\end{proof}


\begin{lemma}\label{IsomAssRedMon}
    Let $(R,\mcQ,\alpha)$ be a log ring and let $\alpha_X^{log}: \mcQ^{log}_X \to \mathcal{O}_X$ be the induced log structure on $X := \Spec(R)$.
    Fix a point $x:=\fp \in \Spec(R)$.
    Then the associated reduced monoid $(\mcQ_{X,x}^{log})_{\textnormal{red}}$ is isomorphic to $(\mcQ_{\alpha^{-1}(\fp)})_{\textnormal{red}}$.
\end{lemma}

\begin{proof}
    The equality (\ref{StalkIsom}) induces $(\mcQ_{X,x}^{log})_{\textnormal{red}} = (\mcQ \oplus_{\alpha^{-1}(R_\fp^\times \cap R)} R^{\times}_\fp)_{\textnormal{red}}$.
    Note that $\alpha^{-1}(\fp) = \mcQ \setminus \alpha^{-1}(R^\times_\fp \cap R)$.
    Then we have a homomorphism $\mcQ_{\alpha^{-1}(\fp)} \to \mcQ_{X,x}^{log}$ which sends $a-b$ to $(a,\frac{1}{\alpha(b)})$. 
    This induces an isomorphism $(\mcQ_{\alpha^{-1}(\fp)})_{\textnormal{red}} \xrightarrow{\cong} (\mcQ_{X,x}^{log})_{\textnormal{red}}$.
    Indeed, since any element $[(a,b)] \in (\mcQ_{X,x}^{log})_{\textnormal{red}}$ is equal to $[(a,1)]$ and an element $[a-0] \in (\mathcal{Q}_{\alpha^{-1}(\fp)})_{\textnormal{red}}$ maps to $[(a,1)]$, $(\mcQ_{\alpha^{-1}(\fp)})_{\textnormal{red}} \to (\mcQ_{X,x}^{log})_{\textnormal{red}}$ is surjective.
    Pick elements $[a-0]$ and $[a'-0]$ of $(\mathcal{Q}_{\alpha^{-1}(\fp)})_{\textnormal{red}}$ such that $[(a,1)] = [(a',1)]$.
    By \Cref{LogarithmicUnits}, there exist $u \in \alpha^{-1}(R^\times_\fp \cap R)$ and $v \in R_\fp^\times$ such that $(a,1) = (a',1) + (u,v) = (a'+u,v)$ in $\mcQ_{X,x}^{log}$.
    Moreover, there exist elements $s, s' \in \alpha^{-1}(R^\times_\fp \cap R)$ such that $a+s = a'+u+s'$ in $\mcQ$.
    Note that $u-0$, $s-0$, and $s'-0$ are units in $\mcQ_{\alpha^{-1}(\fp)}$.
    Hence we obtain the following equalities in $(\mcQ_{\alpha^{-1}(\fp)})_{\textnormal{red}}$
    \[
    [a-0] = [(a+s)-0] = [(a'+u+s')-0] =[a'-0],
    \]
    which implies $(\mcQ_{\alpha^{-1}(\fp)})_{\textnormal{red}} \to (\mcQ_{X,x}^{log})_{\textnormal{red}}$ is injective.
\end{proof}

We define the log-regularity of a log ring.
\begin{definition}\label{DefNonlocalLogReg}
    Let $(R,\mcQ,\alpha)$ be a log ring such that $R$ is Noetherian and $\mcQ$ is cancellative, finitely generated, and root closed.
    Then $(R,\mcQ,\alpha)$ is \textit{log-regular} if $(R_\fp, \mcQ_{\alpha^{-1}(\fp)},\alpha_\fp)$ is local log-regular for every $\fp \in \Spec(R)$.
\end{definition}

\begin{remark}
By \Cref{LocalizationProp}, if $\mcQ$ is cancellative, finitely generated, and root closed, then so is $\mcQ_\fp$ for any $\fp\in \Spec(\mcQ)$.
Hence, in the setting of \Cref{DefNonlocalLogReg}, $(\mcQ_{\alpha^{-1}(\fp)})_{\textnormal{red}}$ is finitely generated, cancellative, and root closed for any $\fp \in \Spec(R)$.
\end{remark}

Let $(X,\mathcal{M}_X,\alpha_X)$ be a fine saturated log scheme.
Then it is \textit{log-regular} if $(\mathcal{O}_{X,x}, \mathcal{M}_{X,x},\alpha_{X,x})$ is local log-regular for every point $x \in X$ (\cite[Chapter \textbf{III}, Theorem 1.11.1]{OgusBook}).
We note that if $(R,\mcQ,\alpha)$ is a log ring such that $\mcQ$ is cancellative, finitely generated, and root closed, then $(X:=\Spec(R), \mcQ_{X}^{log}, \alpha_X^{log})$ is a fine saturated log scheme (see \cite[Example 3]{Tho06} and just below \cite[Definition 4]{Tho06}).
The following proposition implies that the log-regularity of a Noetherian log ring is equivalent to that of the associated log scheme.

\begin{proposition}\label{PropLogRegularRingScheme}
    Let $(R,\mcQ,\alpha)$ be a log ring such that $R$ is Noetherian and $\mcQ$ is finitely generated, cancellative, and root closed.
    Set $X:=\Spec(R)$.
    Then $(R,\mcQ,\alpha)$ is log-regular if and only if $(X, \mcQ_{X}^{log},\alpha_{X}^{log})$ is log-regular.
\end{proposition}

\begin{proof}
    Fix a point $x=\fp \in X$.
    Then we obtain the following equalities in $\mathcal{O}_{X,x} = R_\fp$:
    \[
    \begin{array}{ccl}
      I_{\alpha_{X,x}^{log}} & = & \left\langle \alpha_{X,x}^{log}((a,b)) \in \mathcal{O}_{X,x}~|~  (a,b) \in (\mcQ^{log}_{X,x})^+ \right\rangle \\[3mm]
      & = & \left\langle \alpha_{X,x}^{log}((a,1)) \in \mathcal{O}_{X,x} ~|~ a \in \mcQ \setminus \alpha^{-1}(R_\fp^\times \cap R) \right\rangle \\[3mm]
      & = & \left\langle \frac{\alpha(a)}{1} \in R_\fp ~|~ a \in \mcQ \setminus \alpha^{-1}(R_\fp^\times \cap R) \right\rangle \\[3mm]
      & = & \left\langle \alpha_\fp(a-0) \in R_\fp ~|~a-0 \in \alpha^{-1}(\fp)\mcQ_{\alpha^{-1}(\fp)} = (\mcQ_{\alpha^{-1}(\fp)})^+ \right\rangle \\[3mm]
      & = & \left\langle \alpha_\fp(a-b) \in R_\fp ~|~a-0 \in (\mcQ_{\alpha^{-1}(\fp)})^+, 0-b \in (\mcQ_{\alpha^{-1}(\fp)})^\times \right\rangle \\[3mm]
      & = & I_{\alpha_\fp},
    \end{array}
    \]
    where the fourth equality follows from \Cref{LogarithmicUnits}.
    Hence we obtain $R_\fp/I_{\alpha_{\fp}} = \mathcal{O}_{X,x}/I_{\alpha_{X,x}^{log}}$.
    Moreover, by \Cref{IsomAssRedMon}, we obtain $\dim \mcQ_{\alpha^{-1}(\fp)} = \dim (\mcQ_{\alpha^{-1}(\fp)})_{\textnormal{red}} = \dim (\mcQ_{X,x})_{\textnormal{red}} = \dim \mcQ_{X,x}^{log}$.
    Hence the assertion holds.
\end{proof}

\begin{corollary}\label{LogRegMonAlg}
    Let $\mathcal{Q}$ be a cancellative, finitely generated, and root closed monoid such that the quotient group $\mathcal{Q}^{gp}$ is torsionfree.
    Let $A$ be a regular ring that is not necessarily local.
    Let $R := A[\mathcal{Q}]$ be the monoid algebra.
    Let $\iota : \mcQ\to A[\mcQ]$ be the canonical homomorphism.
    Set $X := \Spec(A[\mcQ])$.
    Then $(X, \mcQ^{log}_X, \iota^{log}_X)$ is a log-regular scheme.
\end{corollary}

\begin{proof}
    By \Cref{maintheorem1}, the local log ring $(A[\mcQ]_\fp,\mcQ_{\alpha^{-1}(\fp)},\alpha_\fp)$ is log-regular for every prime ideal
    $\fp\in X$.
    Hence $(A[\mcQ],\mcQ,\iota)$ is log-regular.
    Therefore, by \Cref{PropLogRegularRingScheme}, the associated log scheme $(X,\mathcal{Q}_X^{log}, \iota_X^{log})$ is log-regular.
\end{proof}

\begin{acknowledgements}
    The author is grateful to Ryo Takahashi for suggesting the importance of clarifying typical examples of non-complete log-regular rings.
    The author also thanks the anonymous referees for their careful reading and constructive feedback, which significantly improved the manuscript.
    The author was partially supported by Grant for Basic Science Research Projects from the Sumitomo Foundation Grant number 2402220.
\end{acknowledgements}

\bibliographystyle{alpha}
\bibliography{Ishiro_ref}

\end{document}